\documentclass[]{amsart}

\usepackage{graphicx,color,hyperref}

\usepackage[initials,nobysame]{amsrefs}
\usepackage{amssymb,amsmath}
\usepackage{subcaption}
\usepackage{comment}

\usepackage{mathtools}
\mathtoolsset{showonlyrefs=true}

\numberwithin{equation}{section}

\newtheorem{theorem}{Theorem}[section]

\newtheorem{corollary}[theorem]{Corollary}
\newtheorem{lemma}[theorem]{Lemma}

\newtheorem{conjecture}[theorem]{Conjecture}
\theoremstyle{definition}

\newtheorem*{acknowledgements}{Acknowledgements}
\theoremstyle{remark}
\newtheorem{remark}[theorem]{Remark}

\newcommand{\R}{\mathbf{R}}
\newcommand{\Z}{\mathbf{Z}}

\renewcommand{\S}{\mathbf{S}}

\begin{document}

\title[Immortal area-preserving curvature flows]{Asymptotic circularity of immortal area-preserving curvature flows}

\author[T.~Miura]{Tatsuya Miura}
\address[T.~Miura]{Department of Mathematics, Graduate School of Science, Kyoto University, Kitashirakawa Oiwake-cho, Sakyo-ku, Kyoto 606-8502, Japan}
\email{tatsuya.miura@math.kyoto-u.ac.jp}


\begin{abstract}
For a class of area-preserving curvature flows of closed planar curves, we prove that every immortal solution becomes asymptotically circular without any additional assumptions on initial data. As a particular corollary, every solution of zero enclosed area blows up in finite time. This settles an open problem posed by Escher--Ito in 2005 for Gage's area-preserving curve shortening flow, and moreover extends it to the surface diffusion flow of arbitrary order. We also establish a general existence theorem for nontrivial immortal solutions under almost circularity and rotational symmetry.
\end{abstract}

\maketitle

\section{Introduction}

Geometric flows without maximum principles entail significant difficulties and are sometimes unexpectedly delicate to handle.
Area-preserving type curvature flows for curves are widely studied, but the absence of maximum principles often complicates their analysis.

In this paper, we present a unified theory for a class of area-preserving curvature flows, focusing on the properties of immortal solutions. 
As a specific corollary, we resolve an open problem posed by Escher--Ito in 2005 \cite[Remark 12]{EI05}:

\emph{Does every Gage's area-preserving curve-shortening flow blow up in finite time when the enclosed area is zero?}

\subsection{Area-preserving curvature flows}

Let $\gamma:\S^1\times[0,T)\to\R^2$, $(x,t)\mapsto \gamma(x,t)$, be a smooth one-parameter family of immersed closed planar curves, where $\S^1=\R/\Z$.
We call $\gamma$ an \emph{$m$-th area-preserving curvature flow}, where $m\geq0$, if it evolves according to the geometric flow equation
\begin{equation}\label{eq:m-ACF}\tag{$m$-ACF}
    \partial_t\gamma= (-1)^m\partial_s^{2m}(k-\bar{k})\nu.
\end{equation}
Here $k$ denotes the signed curvature, $\bar{k}$ the average of $k$, and $\nu$ the unit-normal, with the sign convention that a counterclockwise circle has positive $k$ and inward $\nu$.
The arclength derivative along $\gamma$ is defined by $\partial_s:=|\partial_x\gamma|^{-1}\partial_x$.

This class particularly includes two classical equations; the $m=0$ case is a nonlocal second-order equation called the \emph{area-preserving curve shortening flow} introduced by Gage in 1986 \cite{Gage86}
\begin{equation}\label{eq:ACSF}\tag{ACSF}
    \partial_t\gamma = (k-\bar{k})\nu,
\end{equation}
while all the cases of $m\geq1$ are local and of higher order; in particular, the $m=1$ case is the well-known fourth order \emph{surface diffusion flow}
\begin{equation}\label{eq:SDF}\tag{SDF}
    \partial_t\gamma = -k_{ss}\nu,
\end{equation}
introduced by Mullins in 1957 for describing certain phase-interface dynamics.
The $m\geq2$ case was also recently studied by Parkins--Wheeler \cite{PW16}, wherein it is called the \emph{polyharmonic heat flow for curves}.

The parabolicity of equations \eqref{eq:m-ACF} ensures short-time well-posedness for any given smooth initial curve $\gamma(\cdot,0)=\gamma_0$, and we can always define the maximal existence time $T\in(0,\infty]$.
In addition, both $T<\infty$ and $T=\infty$ are possible to occur.
This is in contrast to relevant equations such as the curve shortening flow $\partial_t\gamma=k\nu$ (for which always $T<\infty$), elastic flows $\partial_t\gamma=-(k_{ss}+\frac{1}{2}k^3-\lambda k)\nu$ with $\lambda\geq0$ (for which $T=\infty$) \cite{DKS02}, and Chen's flow $\partial_t\gamma=-(k_{ss}-k^3)\nu$ ($T<\infty$) \cite{CWW23}.
We say that a solution with $T=\infty$ is \emph{immortal}, and otherwise \emph{blows up in finite time}.

The class \eqref{eq:m-ACF} has the common variational properties: the length $L$ is always decreasing, while the signed area $A$ is preserved. 
Also by smoothness the rotation number $N$ is preserved.
Here,
\[
L[\gamma]:=\int_\gamma ds, \quad A[\gamma]:=-\frac{1}{2}\int_\gamma \langle \gamma,\nu \rangle ds, \quad N[\gamma]:=\frac{1}{2\pi}\int_\gamma kds.
\]
Then the curvature average is precisely given by $\bar{k}:=\frac{1}{L}\int_\gamma kds=\frac{2\pi N}{L}$.
Throughout this paper, up to reversing the parameter, we assume without loss of generality that
\[
N\geq0.
\]
Hereafter, let $A_0:=A[\gamma_0]\in\R$ denote the initial area and $N_0:=N[\gamma_0]\geq0$ the rotation number, both preserved along the flow.

\subsection{Asymptotic circularity and blowup criteria}

Our main result is asymptotic circularity of all immortal solutions to \eqref{eq:m-ACF} without any additional assumptions on initial data.

\begin{theorem}[Asymptotic circularity]\label{thm:main}
    Let $\gamma:\S^1\times[0,T)\to\R^2$ be any solution to $\eqref{eq:m-ACF}$, where $m\geq0$.
    Suppose that $T=\infty$.
    Then $\gamma$ smoothly converges to an $N_0$-fold circle exponentially fast, up to reparametrization.
    More precisely,
    \begin{align}
        & L(t)^2\geq 4\pi N_0A_0>0, \label{eq:main_1}\\
        & L(t)^2-4\pi N_0A_0\leq Ce^{-2N_0ct}, \label{eq:main_2}\\
        & \|k-\bar{k}\|_\infty(t) \leq Ce^{-\frac{1}{4}ct}, \label{eq:main_3}\\
        & \|k_{s^j}\|_\infty(t) \leq C_je^{-\frac{1}{4}ct}\quad (j\geq1), \label{eq:main_4}
    \end{align}
    and in addition there exists a limit point $P_\infty\in\R^2$ such that
    \begin{align}
        \left|\frac{1}{L}\int_\gamma \gamma ds - P_\infty\right| \leq Ce^{-\frac{1}{4}ct}. \label{eq:main_5}
    \end{align}
    Here $c=(2\pi/L_0)^{2m+2}$ and $C>0$ depends only on $m$ and $\gamma$, while $C_j>0$ also on $j\geq1$, but all independent of $t\geq0$.
\end{theorem}

\begin{remark}
    Theorem \ref{thm:main} particularly implies that the \emph{parametric isoperimetric inequality} $L^2\geq 4\pi NA$ holds along any immortal solution.
    Recall that this inequality may not hold for general curves (see Figure \ref{fig:limacon}).
    This issue often presents significant challenges in the analysis of geometric flows for multiply winding curves, both inherently and technically, see e.g.\ \cite{Chou03,WK14,MO21}.
\end{remark}

\begin{figure}[htbp]
    \centering
    \includegraphics[width=0.2\linewidth]{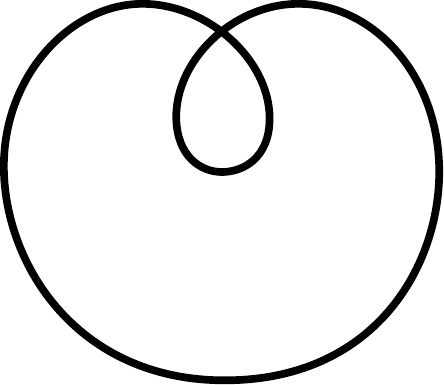}
    \caption{Lima\c{c}on: Counterexample to the parametric isoperimetric inequality.}
    \label{fig:limacon}
\end{figure}

In particular, the contrapositive of \eqref{eq:main_1} directly implies general blowup criteria in terms of initial data only.
Let $L_0:=L[\gamma_0]>0$, the initial length.

\begin{corollary}[Blowup criteria]\label{cor:blowup}
    Let $m\geq0$.
    If either $N_0=0$, $A_0\leq0$, or $L_0^2<4\pi N_0A_0$, then $T<\infty$, i.e., the solution to \eqref{eq:m-ACF} blows up in finite time.
\end{corollary}

\begin{remark}\label{rem:blowup_parameter}
    For $m=0$, our result completely characterizes ``blowup-only parameters'' in the sense that if a parameter $(L_0,A_0,N_0)$ does not meet the assumptions of Corollary \ref{cor:blowup}, or equivalently if $L_0^2\geq4\pi^2N_0A_0>0$, then we can construct an immortal solution to \eqref{eq:ACSF} with the given parameter $(L_0,A_0,N_0)$ as initial data.
    For $m\geq1$, we expect a similar result, but this remains an open problem --- see Section \ref{subsec:existence_immportal} for details.
\end{remark}

Most notably, our blowup criteria fully cover the zero-area case, i.e.,
\[
A_0=0 \Longrightarrow T<\infty,
\]
which was largely underexplored in previous works.
In particular, this result settles Escher--Ito's problem on Gage's flow \eqref{eq:ACSF} \cite[Remark 12]{EI05}.
Not only that, our zero-area criterion is new even for \eqref{eq:SDF} as well as $m\geq2$, whenever $N_0\geq2$.

\begin{itemize}
    \item Case \eqref{eq:ACSF} ($m=0$): If $N_0=0$, then $\bar{k}=0$ and the equation reduces to the standard curve shortening flow, and hence $T<\infty$ (see also \cite{Grayson89}).
    The blowup criteria $A_0<0$ and $L_0^2<4\pi N_0^2A_0^2$ are obtained in \cite[Proposition 9]{EI05}, inspired by \cite{Chou03}.
    The case $A_0=0$ with $N_0\geq1$ was open and nothing but Escher--Ito's problem \cite[Remark 12]{EI05}.
    A partial result towards this problem was obtained in \cite[Theorem 1.1]{WWY18}, which shows that if a locally convex solution has zero area, then either $T<\infty$ or it shrinks to a point as $t\to T=\infty$.
    \item Case \eqref{eq:SDF} ($m=1$):
    The first blowup criterion $N_0=0$ was obtained by Polden \cite[Proposition 3.0.1]{Polden1996} (see also \cite{EGMWW15}).
    Then, Chou found that if either $A_0<0$ or $L_0^2<4\pi N_0^2A_0^2$, then $T<\infty$ \cite[Proposition B]{Chou03}.
    Later, Escher--Ito claimed that if $A_0=0$, then $T=\infty$ for any $N_0\geq1$ \cite[Proposition 11]{EI05}.
    However, their proof is actually only valid for $N_0=1$, as we will demonstrate in Remark \ref{rem:weakly_non-convex}.
    Consequently, the case $A_0=0$ with $N_0\geq2$ has remained open until now.
    \item Case \eqref{eq:m-ACF} ($m\geq2$): To the author's knowledge there are no known criteria, but Polden's criterion $N_0=0$ and Escher--Ito's criterion $A_0=0$ with $N_0=1$ directly extend to this case (see Remarks \ref{rem:blowup_zero_rotation} and \ref{rem:weakly_non-convex}).
\end{itemize}
Our result solves these open cases, and in addition recovers all the above criteria through an independent argument.

The main reason why the zero-area case remained elusive is as follows:
In the $A_0\neq0$ case, all the existing arguments for blowup criteria rely on independently obtained smooth convergence results.
These convergence results are fundamentally based on the $A_0\neq0$ assumption, as it directly implies the non-degeneracy of length along the flow thanks to the isoperimetric inequality $L^2\geq 4\pi|A|=4\pi|A_0|>0$.
This property is essential to proceed with any interpolation-based energy methods \cite{DKS02,Chou03,W13,PW16,NN19,nakamura2019large,NN21}.
For this reason, in the $A_0=0$ case, all the prior arguments focused on directly obtaining finite life-span estimates under stronger assumptions than just $A_0=0$ (as in Remarks \ref{rem:blowup_zero_rotation} and \ref{rem:weakly_non-convex}).

Our argument here overcomes this subtle point by breaking away from the common methodologies discussed above. 
The key idea is a simple yet previously unnoticed change in perspective: we carefully select a sequence of steps, and directly establish ``lower order circularity'' of immortal solutions.
As for exponential smooth convergence, similar results are discussed in \cite{NN19,NN21,Chou03,Wheeler22,nakamura2019large}, all assuming $A_0\neq0$ as explained above, though each may involve additional assumptions.
Theorem \ref{thm:main} unifies these results by a different approach, not only removing the assumption $A_0\neq0$ (see Section \ref{subsec:exp_convergence} for detailed comparisons).

\subsection{Existence of immortal solutions}

Finally we discuss the existence of nontrivial immortal solutions, to which Theorem \ref{thm:main} is indeed applicable.

In addition to Corollary \ref{cor:blowup}, we can easily deduce that $T<\infty$ holds even if $L_0^2=4\pi N_0A_0$ unless the initial curve is an $N_0$-fold circle (so necessarily $N_0\geq2$), which is the only stationary solution.
Hence $L_0^2>4\pi N_0A_0>0$ is necessary for nontrivial immortal solutions.
In this case, it is known that under various additional assumptions, there indeed exist nontrivial immortal solutions for \eqref{eq:m-ACF} or part thereof
\cite{Gage86,Dittberner21,Dittberner2022distance,GP23,WK14,EG97,EMS98,MO21,W13,PW16} (see Section \ref{subsec:existence_immportal} for details).

For completeness, we prove an existence theorem for all \eqref{eq:m-ACF}, particularly extending results in \cite{EG97,W13,MO21,PW16}.
Here the main assumptions on initial curves are closeness to circles and, if $N_0\geq2$, suitable rotational symmetry.
The symmetry assumption is not removable since multiply-covered circles are dynamically unstable.
This idea was used in \cite{MO21} (inspired by \cite{Chou03}) for the first rigorous proof of the existence of a numerically conjectured solution by Escher--Mayer--Simonett \cite{EMS98}.

To state the theorem we recall the scale-invariant \emph{curvature oscillation}:
\[
    K_{osc}:=L\int_\gamma(k-\bar{k})^2ds = L\int_\gamma k^2ds - 4\pi^2 N^2.
\]
In addition, we define the \emph{Abresch--Langer class} $A_{\ell,n}$ \cite{AL86} as all curves $\gamma:\S^1\to\R^2$ with $N[\gamma]=\ell$ such that $\gamma(x+\frac{1}{n})=R_{2\pi\ell/n}\gamma(x)$ holds for $x\in\S^1$, where $R_\theta$ denotes the counterclockwise rotation matrix through angle $\theta$ (see also Figure \ref{fig:area_zero}).

\begin{figure}[htbp]
    \centering
    \includegraphics[width=0.8\linewidth]{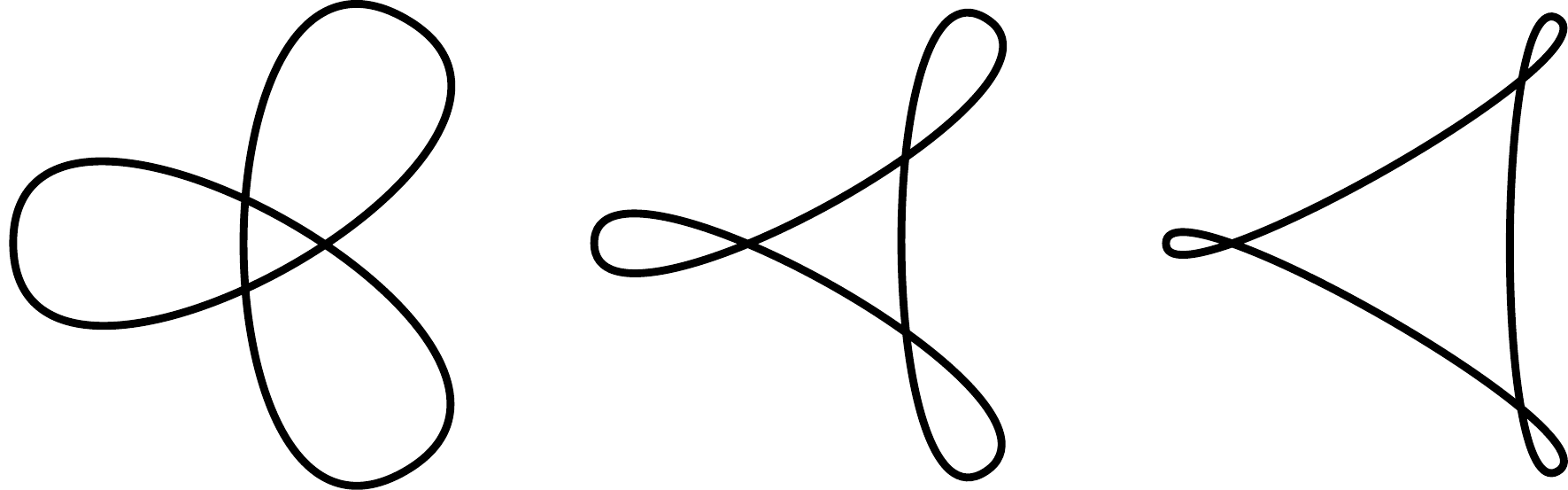}
    \caption{Trefoil: Locally convex curves of class $A_{2,3}$.}
    \label{fig:area_zero}
\end{figure}

\begin{theorem}[Existence of immortal solutions]\label{thm:rotationary_symmetric}
    Let $\gamma:\S^1\times[0,T)\to\R^2$ be a solution to $\eqref{eq:m-ACF}$, where $m\geq0$.
    Suppose that the initial curve $\gamma_0$ satisfies
    \[
    K_{osc}[\gamma_0]\leq \varepsilon_1, \quad L_0^2 \leq (1+\varepsilon_2)4\pi N_0 A_0,
    \]
    where $\varepsilon_1,\varepsilon_2>0$ are constants depending only on $m$ and $N_0$,
    and in addition that $\gamma_0\in A_{N_0,n}$ for some $n\geq N_0$.
    Then $T=\infty$, i.e., the solution from $\gamma_0$ is immortal.
\end{theorem}

Our result is new when $N_0\geq2$, except for the $m=1$ case treated in \cite{MO21}.
The above class of initial curves always contains non-convex curves.
Taking smaller $\varepsilon_1$ if necessary, we may remove the lower order assumption involving $\varepsilon_2$.

This paper is organized as follows.
In Section \ref{sec:prelim} we prepare notation and recall fundamental facts, and then discuss some known arguments for blowup to clarify how the zero-area case is delicate.
In Section \ref{sec:condition_immortality} we mainly prove Theorem \ref{thm:main}, and also Theorem \ref{thm:rotationary_symmetric} in order, with more precise comparisons to previous results.

\begin{acknowledgements}
    The author would like to thank \mbox{Fuya} \mbox{Hiroi} and \mbox{Shinya} \mbox{Okabe} for fruitful discussions, and in particular letting him know the technical gap of Escher--Ito's proof with Hiroi's counterexample.
    The author is also grateful to \mbox{Glen} \mbox{Wheeler} for continuous discussions.
    This work is supported by JSPS KAKENHI Grant Numbers JP21H00990, JP23H00085, and JP24K00532.
\end{acknowledgements}

\section{Preliminaries}\label{sec:prelim}

We first prepare notation.
For $f:\S^1\times[0,T)\to\R$ we denote $\|f\|_p:=\int_\gamma|f|^pds=\int_{\S^1}|f|^p|\gamma_x|dx$ for $p\in(1,\infty)$, while $\|f\|_\infty:=\max_{\S^1}|f|$, at each time slice $t\in[0,T)$.
We will often drop the integral domain $\gamma$ if it is clear.

We then recall standard general formulae for the equation of the form
\[
\partial_t\gamma =-F\nu.
\]
For the arclength measure $ds:=|\partial_x\gamma|dx$ we directly compute
\begin{align*}
    ds_t &= \frac{\langle\gamma_{xt},\gamma_{x}\rangle}{|\gamma_x|}dx = \frac{\langle(-F\nu)_x,\gamma_{x}\rangle}{|\gamma_x|}dx = -F\frac{|\gamma_x|\langle-k\tau,\gamma_{x}\rangle}{|\gamma_x|}dx = Fkds.
\end{align*}
For other computations it is convenient to introduce the commutator
\[
    [\partial_t,\partial_s]
    = -Fk\partial_s.
\]
For example, applying this to the tangent vector $\tau:=\gamma_s$ as well as $\nu$ we compute
\begin{align*}
\tau_t
&= \partial_t\partial_s \gamma = [\partial_t,\partial_s] \gamma + \partial_s\partial_t\gamma = -Fk\tau + (-F\nu)_s = -F_s\nu,\\
\nu_t &= F_s\tau.
\end{align*}
Now, using our equations \eqref{eq:m-ACF} where
\[
F = (-1)^{m+1}(k-\bar{k})_{s^{2m}}, 
\]
we are already able to compute the derivative formulae for the area,
\begin{align*}
    \frac{d}{dt}A &= -\frac{1}{2}\int (\langle \gamma_t,\nu \rangle + \langle \gamma,\nu_t \rangle) ds - \frac{1}{2}\int \langle \gamma,\nu \rangle ds_t\\
    & = -\frac{1}{2}\int (-F +  F_s\langle \gamma,\tau \rangle + F\langle \gamma,k\nu\rangle ) ds \\
    &= -\frac{1}{2}\int (-F -  F\langle \gamma,\gamma_{s} \rangle_s + F\langle \gamma,\gamma_{ss}\rangle ) ds \\
    & = \int Fds = (-1)^{m+1}\int (k-\bar{k})_{s^{2m}} ds = 0,
\end{align*}
as well as for the length
\begin{align*}
    \frac{d}{dt}L = \int ds_t &= (-1)^{m+1}\int k(k-\bar{k})_{s^{2m}}ds \\
    &= (-1)^{m+1}\int (k-\bar{k})(k-\bar{k})_{s^{2m}} ds \\
    &= -\int (k-\bar{k})_{s^m}^2ds = -\|(k-\bar{k})_{s^m}\|_2^2 \leq0.
\end{align*}
The same formulae are also obtained in \cite{PW16,nakamura2019large}.
Notice in particular that along the flow \eqref{eq:m-ACF} we always have
\[
0\leq L(t)\leq L_0, \quad A(t)=A_0, \quad N(t)=N_0.
\]

We also recall the following elementary $L^2$-estimates
\begin{equation}\label{eq:PW}
    \|(k-\bar{k})_{s^m}\|_2^2\leq \frac{L^2}{4\pi^2}\|(k-\bar{k})_{s^{m+1}}\|_2^2 \quad (m\geq0),
\end{equation}
or equivalently, 
\[
\|k-\bar{k}\|_2^2\leq \frac{L^2}{4\pi^2}\|k_s\|_2^2 \quad\text{and}\quad \|k_{s^m}\|_2^2\leq \frac{L^2}{4\pi^2}\|k_{s^{m+1}}\|_2^2\quad (m\geq1).
\]
These are easily obtained by the Poincar\'e--Wirtinger inequality since each $(k-\bar{k})_{s^m}$ is $L$-periodic and average-free.
We also have $L^\infty$-counterparts \cite[Corollary 2.3]{W13}
\begin{equation}\label{eq:PW_infty}
    \|(k-\bar{k})_{s^m}\|_\infty^2\leq \frac{L}{2\pi}\|(k-\bar{k})_{s^{m+1}}\|_2^2 \quad (m\geq0).
\end{equation}

We now observe that Polden's proof \cite{Polden1996} for \eqref{eq:SDF} also directly works for any \eqref{eq:m-ACF} with $m\geq0$.

\begin{remark}[Blowup of solutions with $N_0=0$]\label{rem:blowup_zero_rotation}
    Suppose that $N_0=0$.
    Then $\bar{k}=0$.
    By \eqref{eq:PW} and the derivative formula for $L$, we have
    \[
    L'=-\int (k-\bar{k})_{s^m}^2ds \leq -\left(\frac{4\pi^2}{L^2}\right)^m\int (k-\bar{k})^2ds = -\left(\frac{4\pi^2}{L^2}\right)^m\int k^2ds.
    \]
    Now we use the Cauchy--Schwarz inequality and Fenchel's theorem,
    \[
    \int k^2ds \geq \frac{1}{L}\left(\int|k|ds\right)^2 \geq \frac{4\pi^2}{L},
    \]
    to obtain $L' \leq -(4\pi^2)^{m+1}L^{-(2m+1)}$
    or finally $(L^{2m+2})'\leq -(2m+2)(4\pi^2)^{m+1}$. 
    This implies finite time blowup $T<\infty$.
\end{remark}

We also observe that Eschler--Ito's argument \cite{EI05} for weakly non-convex solutions to \eqref{eq:SDF} also works for any \eqref{eq:m-ACF} with $m\geq1$ (but not for $m=0$).
Here a solution $\gamma:\S^1\times[0,T)\to\R^2$ is said to be \emph{weakly non-convex} if the curvature has a zero along the flow, i.e., for any $t\in[0,T)$ there is $x_t\in\S^1$ such that $k(x_t,t)=0$.
A solution is not weakly non-convex if and only if it is strictly locally convex at some time.

\begin{remark}[Blowup of weakly non-convex solutions]\label{rem:weakly_non-convex}
    Let $m\geq1$.
    By \eqref{eq:PW} and the derivative formula for $L$, we have
    \[
    L'=-\int k_{s^m}^2ds \leq -\left(\frac{4\pi^2}{L^2}\right)^{m-1}\int k_s^2ds.
    \]
    If the solution is weakly non-convex, then a version of the Wirtinger inequality (which does not require $k$ to be average-free) implies that $\|k\|_2^2\leq\frac{L^2}{\pi^2}\|k_s\|_2^2$.
    Hence $T<\infty$ follows similarly to Remark \ref{rem:blowup_zero_rotation}.

    Escher--Ito \cite[Proposition 11]{EI05} (also \cite[Lemma 2.2]{Wheeler22}) claimed for $N_0\geq1$ that if $A_0=0$ then the solution is always weakly non-convex.
    This claim is indeed valid for $N_0=1$, since every locally convex curve with $N_0=1$ is embedded, so the area must be positive.
    However, the $N_0\geq2$ case has counterexamples; there exist strictly locally convex curves with zero or negative area (see Figure \ref{fig:area_zero}), which the author learned about through personal communication with Fuya Hiroi.
    Such examples are also previously known in different contexts, see e.g.\ \cite{AL86,WK14}.
\end{remark}

We will anyway not use the above criteria, and our independent argument also works for those cases.

Finally, we record the well-known fact that if a solution to \eqref{eq:m-ACF} blows up in finite time ($T<\infty$), then the $L^2$-norm of the curvature blows up as follows:
\begin{equation}\label{eq:blowup_rate}
    T<\infty \Longrightarrow \int_\gamma k^2ds \geq \frac{C}{(T-t)^{\frac{1}{2m+2}}}.
\end{equation}
In particular, if this $L^2$-norm remains bounded on $[0,T)$, then the solution is immortal ($T=\infty$).
The case of $m=1$ is first obtained in \cite[Theorem 3.1]{DKS02}.
See also \cite[Proposition 5]{EI05} and \cite[Theorem 3.2]{NN21} for $m=0$, and \cite[Lemma 11]{PW16}
for $m\geq2$.

\section{Immortal solutions}\label{sec:condition_immortality}

In this section we prove all the theorems in the introduction.

\subsection{Asymptotic isoperimetry}\label{subsec:asymptotic_circularity}

We first prove \eqref{eq:main_1} and also a (qualitative) convergence of the isoperimetric defect.

\begin{lemma}\label{lem:main}
    Let $\gamma:\S^1\times[0,T)\to\R^2$ be any solution to $\eqref{eq:m-ACF}$, where $m\geq0$.
    Suppose that $T=\infty$.
    Then $N_0>0$ and $A_0>0$ need to hold, and $L(t)$ monotonically decreases to $\sqrt{4\pi N_0A_0}$ as $t\to\infty$.
    In addition, there is a sequence $t_j\to\infty$ such that $K_{osc}(t_j)\to0$ as $j\to\infty$.
\end{lemma}

\begin{proof}
    Since $0 \leq L(t)\leq L_0$ and $L'=-\|(k-\bar{k})_{s^{m}}\|_{2}^2$, we have $\|(k-\bar{k})_{s^{m}}\|_{2}^2\in L^1(0,\infty)$ and hence there is a sequence $t_j\to\infty$ such that
    \[
    \|(k-\bar{k})_{s^{m}}\|_{2}(t_j)\to0.
    \]
    In particular, using \eqref{eq:PW} we have
    \[
    K_{osc} \leq L\left(\frac{L^2}{4\pi^2}\right)^m \|(k-\bar{k})_{s^{m}}\|_{2}^2 \leq L_0\left(\frac{L_0^2}{4\pi^2}\right)^m \|(k-\bar{k})_{s^{m}}\|_{2}^2, 
    \] 
    and thus, regardless of the value of $m\geq0$, 
    \begin{equation}\label{eq:K_osc_decay}
        K_{osc}(t_j)\to0.
    \end{equation}

    Let $\tilde{\gamma}$ be the arclength reparametrization of $\gamma$, and let 
    \[
    \eta_j(x):=\frac{1}{L}\tilde{\gamma}(Lx,t_j), \quad x\in \S^1,
    \]
    be the normalized curve of unit length $L[\eta_j]=1$.
    Note that $N[\eta_j]=N_0$.
    The scale-invariance of $K_{osc}$ and \eqref{eq:K_osc_decay} imply that the curvature $k^{\eta_j}$ of $\eta_j$ satisfies
    \begin{equation}\label{eq:rescale_decay}
        \|k^{\eta_j}-2\pi N_0\|_2^2 = K_{osc}[\eta_j] \to0.
    \end{equation}
    Let $\theta^{\eta_j}:[0,1]\to\R$ denote the (smooth) tangential angle function.
    Then $\eta_j'=\partial_s\eta_j=(\cos\theta^{\eta_j},\sin\theta^{\eta_j})$ and $(\theta^{\eta_j})'=\partial_s\theta^{\eta_j}=k^{\eta_j}$.
    Up to shifting the parameter we may assume $\theta^{\eta_j}(0)=0$.
    Define
    \[
    f_j(x):=\theta^{\eta_j}(x)-2\pi N_0x.
    \]
    Then by \eqref{eq:rescale_decay} we obtain $f_j'\to0$ in $L^2(0,1)$, and also by $|f_j(x)|\leq |f(0)|+\int_0^x|f_j'| \leq (\int_0^1|f_j'|^2)^{1/2}$ we further deduce that 
    \[
    \sup_{x\in[0,1]}|f_j(x)| \to 0.
    \]
    Hence the shifted curve $\hat{\eta}_j:=\eta_j-\eta_j(0)$ converges to $\eta_\infty$ defined by
    \[
    \eta_\infty(x):=\int_0^x(\cos{2\pi N_0z},\sin{2\pi N_0z})dz
    \]
    in the $C^1$-topology on $\S^1$ (since $\hat{\eta}_j(0)=\eta_\infty(0)$ and $\hat{\eta}_j'\to \eta_\infty'$ uniformly).

    If $N_0=0$, then $\eta_\infty$ is a segment, which is impossible as a $C^1$-limit of closed curves; therefore, we necessarily have
    \begin{equation}\label{eq:rot_number_nonzero}
        N_0>0.
    \end{equation}
    Thus $\eta_\infty$ is nothing but an $N_0$-fold circle of unit length, meaning that 
    $
    A[\eta_\infty]=\frac{1}{4\pi N_0}>0.
    $
    By $C^1$-convergence we have $A[\eta_j]=A[\hat{\eta}_j]\to A[\eta_\infty]$, so in particular
    \[
        A[\eta_j]>0 \quad (j\gg1).
    \]
    Also, the isoperimetric ratio
    $
    I := \frac{L^2}{4\pi A} 
    $
    can now be well-defined for $\eta_j$ with all large $j$, and satisfies
    \[
    I[\eta_j] \to N_0.
    \]
    
    Since the above two properties are preserved under rescaling (and other geometric invariances like parameter-shifts), we can go back to $\gamma$ to deduce that
    \[
    A(t_j)=A[\gamma(\cdot,t_j)]>0 \quad (j\gg1), \quad I(t_j)=I[\gamma(\cdot,t_j)]\to N_0.
    \]
    In particular, since the area is preserved along the flow $\gamma$,
    \begin{equation}\label{eq:area>0}
        A_0=A(t)>0 \quad (t\geq 0),
    \end{equation}
    and since the isoperimetric ratio is decreasing, we have the monotone convergence
    \begin{equation}\label{eq:iso_def}
        I(t) = \frac{L(t)^2}{4\pi A_0} \searrow N_0 \quad (t\to\infty).
    \end{equation}
    The proof is now complete with \eqref{eq:K_osc_decay}, \eqref{eq:rot_number_nonzero}, \eqref{eq:area>0}, and \eqref{eq:iso_def}.
\end{proof}

Corollary \ref{cor:blowup} is a direct consequence of Lemma \ref{lem:main}.

\subsection{Exponential smooth convergence}\label{subsec:exp_convergence}

Now we discuss exponential smooth convergence, completing the proof of Theorem \ref{thm:main}.
Since we are interested in immortal solutions, hereafter we assume that
\[
N_0\geq1.
\]

Our argument here is based on a curvature-oscillation approach, as in previous works \cite{W13, PW16}. 
However, we need to adapt the argument, as those works contain technical gaps in deriving the initial exponential convergence. 
Instead, we begin by recalling the general estimate for closed curves given in \cite[Lemma 2.5]{NN21},
\begin{equation}\label{eq:NN}
    \frac{4\pi^2 N}{L^2}|D| \leq K_{osc},
\end{equation}
where, here and throughout, $D$ denotes the \emph{isoperimetric defect}
\begin{equation}
    D:= L^2-4\pi NA.
\end{equation}
We use this estimate to establish the exponential decay of $D$.

\begin{lemma}\label{lem:initial_exp_decay}
    For any immortal solution to \eqref{eq:m-ACF} with $m\geq0$, the estimate
    \[
    0 \leq D(t) \leq D(0)e^{-2N_0ct}
    \]
    holds for $t\geq0$, where $c:=(2\pi/L_0)^{2m+2}$.
\end{lemma}

\begin{proof}
    Lemma \ref{lem:main} ensures $D\geq0$, and also $N_0\geq1$.
    We compute $D' = 2LL' = -2L\|(k-\bar{k})_{s^m}\|_2^2$.
    Using \eqref{eq:PW} and \eqref{eq:NN} as well as $L\leq L_0$, we estimate
    \begin{align}
        D' \leq -2L\left(\frac{4\pi^2}{L^2}\right)^m\|k-\bar{k}\|_2^2 = -2\left(\frac{4\pi^2}{L^2}\right)^m K_{osc} \leq -2N_0cD.
    \end{align}
    This implies the desired exponential decay of $D$.
\end{proof}

For proving the exponential decay of higher order quantities, we proceed with the following calculations for the general flow $\partial_t\gamma=-F\nu$, where we define $\kappa:=\gamma_{ss}$:
\begin{align*}
\kappa_t
&= [\partial_t,\partial_s] \tau + \partial_s\partial_t \tau
= -Fk\kappa + (-F_s\nu)_s
= -(F_{ss} + Fk^2)\nu + F_sk\tau,\\
k_t &= \langle\kappa,\nu\rangle_t
= -(F_{ss} + Fk^2).
\end{align*}
In particular, for \eqref{eq:m-ACF} with
$F = (-1)^{m+1}(k-\bar{k})_{s^{2m}}$,
\begin{align*}
k_t &= (-1)^m\left( k_{s^{2m+2}} + k^2(k-\bar{k})_{s^{2m}} \right).
\end{align*}
Then
\begin{align*}
    \frac{d}{dt}\int(k-\bar{k})^2ds &= 2\int (k-\bar{k}) k_t ds + \int (k-\bar{k})^2 ds_t\\
    &= 2(-1)^m\int_\gamma (k-\bar{k}) \left( k_{s^{2m+2}} + k^2(k-\bar{k})_{s^{2m}} \right) ds\\
    & \qquad + (-1)^{m+1}\int (k-\bar{k})^2 k(k-\bar{k})_{s^{2m}} ds\\
    &= 2(-1)^m\int (k-\bar{k}) k_{s^{2m+2}} ds \\
    & \qquad + (-1)^m\int  \big( (k-\bar{k})^3+3\bar{k}(k-\bar{k})^2\big) (k-\bar{k})_{s^{2m}} ds \\
    & \qquad \qquad + (-1)^m2\bar{k}^2 \int (k-\bar{k}) (k-\bar{k})_{s^{2m}} ds.
\end{align*}
Letting
\[
Q:= (k-\bar{k})^3+3\bar{k}(k-\bar{k})^2
\]
and integrating by parts, we obtain
\[
\frac{d}{dt}\int(k-\bar{k})^2ds = -2\int k_{s^{m+1}}^2ds  + \int Q_{s^m}(k-\bar{k})_{s^{m}} ds +  2\bar{k}^2\int(k-\bar{k})_{s^m}^2ds.
\]
(This reproduces \cite[Lemma 6]{PW16} fixing typos.)
Then we obtain an estimate for the curvature oscillation as follows:
\begin{align}
    \frac{d}{dt}K_{osc}&= \frac{dL}{dt}\int(k-\bar{k})^2ds + L\frac{d}{dt}\int(k-\bar{k})^2ds \notag\\
    &\leq L\frac{d}{dt}\int(k-\bar{k})^2ds \notag\\
    &= -2L\int k_{s^{m+1}}^2ds  + L\int Q_{s^m}(k-\bar{k})_{s^{m}} ds -  8\pi^2N_0^2\frac{d}{dt}\log{L}. \label{eq:K_osc}
\end{align}

Now, in order to estimate the term involving $Q$, we recall a well-known interpolation estimate of Dziuk--Kuwert--Sch\"atzle style \cite{DKS02}:

\begin{lemma}\label{lem:interpolation}
    Let $P_\nu^{\mu,\lambda}=P_\nu^{\mu,\lambda}(k-\bar{k})$ denote a linear combination (with universal coefficients) of terms of the form $\partial_s^{i_1}(k-\bar{k})\cdots \partial_s^{i_\nu}(k-\bar{k})$, where $\mu=i_1+\dots+i_\nu\geq2$ and $\lambda=\max\{i_1,\dots,i_\nu\}+1$.
    Then
    \[
    \int |P_\nu^{\mu,\lambda}| ds \leq C(\mu,\nu,\lambda) L^{1-\mu-\nu} (K_{osc})^{\frac{1}{2}(\nu-\gamma)}\left(L^{2\lambda+1}\int k_{s^\lambda}^2ds\right)^{\frac{1}{2}\gamma}
    \]
    where $\gamma=(\mu+\frac{1}{2}\nu-1)/\lambda$, and $C(\mu,\nu,\lambda)>0$ depends only on $\mu,\nu,\lambda$.
\end{lemma}

\begin{remark}
    The proof of Lemma \ref{lem:interpolation} is almost same as \cite[Proposition 2.5]{DKS02}; the only difference is that we replace $\varkappa$ ($=k\nu$) with $(k-\bar{k})\nu$ (see also \cite[Lemma 21]{PW16}).
    Recall also that this $P$-style term obeys the elementary rules
    \[
    P_\nu^{\mu,\lambda}P_{\nu'}^{\mu',\lambda'}=P_{\nu+\nu'}^{\mu+\mu',\max\{\lambda,\lambda'\}}, \quad \partial_sP_\nu^{\mu,\lambda}=P_\nu^{\mu+1,\lambda+1}.
    \]
\end{remark}

Using this lemma, we obtain a key estimate for the curvature oscillation.

\begin{lemma}\label{lem:K_osc_estimate}
    Let $\gamma:\S^1\times[0,T)\to\R^2$ be any solution to $\eqref{eq:m-ACF}$, where $m\geq0$ and $N_0\geq1$.
    Then there are $c_1,c_2>0$ depending only on $m$ such that
    \[
    \frac{d}{dt}(K_{osc}+8\pi^2N_0^2\log{L}) + L\left( 2-c_1K_{osc}-c_2N_0\sqrt{K_{osc}} \right)\|k_{s^{m+1}}\|_2^2 \leq 0.
    \]
\end{lemma}

\begin{proof}
    Hereafter $C>0$ denotes a constant depending only on $m$ and may change line by line.
    Since $Q=(k-\bar{k})^3+3\bar{k}(k-\bar{k})^2=P_{3}^{0,1} + \bar{k}P_{2}^{0,1}$, we have
    \[
    Q_{s^m}(k-\bar{k})_{s^{m}}=(P_{3}^{m,m+1} + \bar{k}P_{2}^{m,m+1})P_1^{m,m+1}=P_{4}^{2m,m+1} + \bar{k}P_{3}^{2m,m+1}.
    \]
    By using Lemma \ref{lem:interpolation} with $\nu=4$, $\mu=2m$, $\lambda=m+1$, where $\gamma=2-\frac{1}{m+1}$, and then using $K_{osc}\leq CL^{2m+3}\|k_{s^{m+1}}\|_2^2$ obtained from \eqref{eq:PW}, we deduce that
    \begin{align*}
        \int |P_{4}^{2m,m+1}| ds &\leq CL^{-2m-3}(K_{osc})^{1+\frac{1}{2m+2}}\left(L^{2m+3}\int k_{s^{m+1}}^2ds\right)^{1-\frac{1}{2m+2}}\\
        & \leq CL^{-2m-3}K_{osc}\left(L^{2m+3}\int k_{s^{m+1}}^2ds\right).
    \end{align*}
    Similarly, by using Lemma \ref{lem:interpolation} and \eqref{eq:PW}, 
    we obtain
    \begin{align*}
        \bar{k}\int |P_{3}^{2m,m+1}| ds &\leq \bar{k}CL^{-2m-2}(K_{osc})^{\frac{1}{2}+\frac{3}{4m+4}}\left(L^{2m+3}\int k_{s^{m+1}}^2ds\right)^{1-\frac{3}{4m+4}}\\
        & \leq CN_0L^{-2m-3}\sqrt{K_{osc}}\left(L^{2m+3}\int k_{s^{m+1}}^2ds\right).
    \end{align*}
    In summary, we establish the estimate of the form
    \[
    \int Q_{s^m}(k-\bar{k})_{s^m} \leq C(K_{osc}+N_0\sqrt{K_{osc}})\|k_{s^{m+1}}\|_2^2.
    \]
    Inserting this into \eqref{eq:K_osc} completes the proof.
\end{proof}

Now we obtain a sufficient condition for $K_{osc}$-smallness preservation, which automatically ensures immortality, extending similar results in \cite{W13,PW16,MO21}.

\begin{lemma}\label{lem:global_existence}
    Let $\gamma:\S^1\times[0,T)\to\R^2$ be any solution to $\eqref{eq:m-ACF}$, where $m\geq0$ and $N_0\geq1$.
    Then there is a constant $\bar{\varepsilon}_1=\bar{\varepsilon}_1(m,N_0)>0$ with the following property:
    For any $\varepsilon_1\in(0,\bar{\varepsilon}_1]$ there exists $\varepsilon_2=\varepsilon_2(m,N_0,\varepsilon_1)>0$ such that, if at some $t_0\in[0,T)$
    \begin{equation}\label{eq:epsilon_small}
        K_{osc}(t_0)\leq \varepsilon_1, \qquad L(t_0)\leq (1+\varepsilon_2)\inf_{t\in[t_0,T)}L(t),
    \end{equation}
    then $T=\infty$ and $K_{osc}(t)\leq 2\varepsilon_1$ holds for all $t\geq t_0$.
\end{lemma}

\begin{proof}
    Let $c_1,c_2>0$ be as in Lemma \ref{lem:K_osc_estimate}.
    Let $x^*>0$ be the first positive root of $2-c_1x-c_2N_0\sqrt{x}$, and let 
    \[
    \bar{\varepsilon}_1:=\frac{1}{2}x^*.
    \]
    Now we assume \eqref{eq:epsilon_small} with any $\varepsilon_1\in(0,\bar{\varepsilon}_1]$ (and $\varepsilon_2>0$ specified later).
    Let $T'\in(0,T]$ be the maximal time such that $K_{osc}\leq 2\varepsilon_1$ holds on $[t_0,T')$.
    Then in particular $K_{osc}\leq x^*$ on $[t_0,T')$, so by integrating the estimate in Lemma \ref{lem:K_osc_estimate} we obtain
    \[
    K_{osc}(t) \leq K_{osc}(t_0) + 8\pi^2N_0^2\log\frac{L(0
    )}{\inf_{t\in[t_0,T')}L(t)}, \quad t\in[t_0,T').
    \]
    Hence, if we choose $\varepsilon_2$ so small that 
    \[
    8\pi^2N_0^2\log(1+\varepsilon_2)\leq\frac{1}{2}\varepsilon_1,
    \]
    then by \eqref{eq:epsilon_small} we deduce that $K_{osc}\leq\frac{3}{2}\varepsilon_1$ holds on $[t_0,T')$, meaning that $T'=T$ due to the maximality of $T'$.
    Consequently,
    \[
    \sup_{t\in[0,T)}K_{osc}(t) \leq 2\varepsilon_1.
    \]
    In addition, this with the length estimate in \eqref{eq:epsilon_small} further implies that
    \[
    \sup_{t\in[0,T)}\|k\|_2^2= \sup_{t\in[0,T)}\frac{1}{L}(K_{osc}+4\pi^2N_0^2) <\infty,
    \]
    which together with \eqref{eq:blowup_rate} verifies $T=\infty$.
\end{proof}

We now prove the remaining exponential convergences in Theorem \ref{thm:main}.

\begin{proof}[Proof of Theorem \ref{thm:main}]    
    Lemma \ref{lem:main} already ensures \eqref{eq:main_1}, while Lemma \ref{lem:initial_exp_decay} implies \eqref{eq:main_2}, so we only address \eqref{eq:main_3}, \eqref{eq:main_4}, and \eqref{eq:main_5}.
    
    By Lemma \ref{lem:main}, for any small $\varepsilon>0$ we can find a large $t_\varepsilon>0$ that satisfies the assumption of Lemma \ref{lem:global_existence} with $\varepsilon_1:=\varepsilon$ and $t_0:=t_\varepsilon$.
    This particularly implies that $K_{osc}$ converges to $0$ as $t\to\infty$.
    
    Now we fix a large $t_*>0$ such that $2 - c_1K_{osc}-c_2N_0\sqrt{K_{osc}}\geq1$ on $[t_*,\infty)$.
    Then by using Lemma \ref{lem:K_osc_estimate} with \eqref{eq:PW} and $L(t)\leq L_0$, we deduce that, for $t\in [t_*,\infty)$,
    \begin{align}
       \frac{d}{dt}K_{osc} + \left(\frac{4\pi^2}{L_0^2}\right)^{m+1}K_{osc} \leq - 8\pi^2N_0^2 \frac{d}{dt}\log{L}.
    \end{align}
    Noting that $(\log L)'=\frac{1}{2L^2}D'\leq0$ and $L^2\geq 4\pi N_0A_0>0$, we further obtain
    \begin{align}
       \frac{d}{dt}K_{osc} + \left(\frac{4\pi^2}{L_0^2}\right)^{m+1}K_{osc} \leq -\frac{\pi N_0}{A_0}\frac{d}{dt}D.
    \end{align}
    Multiplying $e^{ct}$ with $c:=(2\pi/L_0)^{2m+2}$ and integrating over $[\frac{t}{2},t]\subset[t_*,\infty)$ yield
    \begin{align}
        e^{ct}K_{osc}(t) - e^{\frac{1}{2}ct}K_{osc}(\tfrac{t}{2}) &\leq \frac{\pi N_0}{A_0}\int_{\frac{t}{2}}^t e^{ct'}(-\tfrac{d}{dt}D(t'))dt'\\
        &\leq \frac{\pi N_0}{A_0}e^{ct}\int_{\frac{t}{2}}^\infty (-\tfrac{d}{dt}D(t'))dt'\\
        &= \frac{\pi N_0}{A_0}e^{ct}D(\tfrac{t}{2}),
    \end{align}
    where in the last part we also used $D(t)\to0$ as $t\to\infty$ (Lemma \ref{lem:main}).
    Further, multiplying $e^{-ct}$ and using the decay $D(\frac{t}{2})=o(e^{-\frac{1}{2}ct})$ (Lemma \ref{lem:initial_exp_decay}) and the boundedness of $K_{osc}$, we deduce that, even for all $t\geq0$,
    \begin{align}\label{eq:Kosc_exp_decay}
        K_{osc}(t) \leq C e^{-\frac{1}{2}ct},
    \end{align}
    where $C>0$ depends on $m$ and $\gamma$ but not $t$.
    Hereafter $C$ will change line by line.
    
    The decay of higher order terms $\|k_{s^j}\|_2$ now follows by standard interpolation techniques.
    Following the derivation of \cite[Equation (38)]{PW16} (again fixing minor typos as above), we obtain a higher order version of Lemma \ref{lem:K_osc_estimate} for $j\geq1$:
    \[
    \frac{d}{dt}\int k_{s^j}^2ds + \left(2-c_j(K_{osc}+N_0\sqrt{K_{osc}})\right) \int k_{s^{j+m+1}}^2 \leq \frac{8\pi^2N_0^2}{L^2}\int k_{s^{j+m}}^2ds,
    \]
    where $c_j$ depends only on $m$ and $j$, and will change line by line.
    Then for $k_{s^{j+m}}^2=P_{2}^{2j+2m,j+m+1}$ we apply Lemma \ref{lem:interpolation} and \eqref{eq:PW} to further estimate
    \begin{align*}
        \int k_{s^{j+m}}^2ds &\leq c_jL^{-2j-2m-1}(K_{osc})^{\frac{1}{2}+\frac{1}{j+m+1}}\left(L^{2j+2m+3}\int k_{s^{j+m+1}}^2ds\right)^{\frac{1}{2}-\frac{1}{j+m+1}}\\
        &\leq c_jL^{-2j-2m-1}(K_{osc})^{\frac{1}{2}}\left(L^{2j+2m+3}\int k_{s^{j+m+1}}^2ds\right)^{\frac{1}{2}}\\
        &\leq \frac{c_j}{L^{j+m-1}}\left( \delta\|k_{s^{j+m+1}}\|_2^2+\frac{1}{\delta}\|k-\bar{k}\|_2^2\right).
    \end{align*}
    Taking $\delta=(4\pi A_0N_0)^{\frac{j+m+1}{2}}(16\pi^2N_0^2c_j)^{-1}$ and using $L\geq\sqrt{4\pi A_0N_0}$, we obtain
    \[
    \frac{d}{dt}\int k_{s^j}^2ds + \left(\frac{3}{2}-c_j\left(K_{osc}+N_0\sqrt{K_{osc}}\right)\right) \int k_{s^{j+m+1}}^2 \leq C_jK_{osc},
    \]
    where $C_j>0$ also depends on $\gamma$ and will change line by line.
    Then, using the fact that $K_{osc}\to0$ as $t\to\infty$, as well as \eqref{eq:PW} and $L\leq L_0$, we deduce that there is $t_{*j}>0$ such that 
    \[
        \frac{d}{dt}\int k_{s^j}^2ds + \left(\frac{4\pi^2}{L_0^2}\right)^{m+1}\int k_{s^{j}}^2 \leq C_jK_{osc}, \quad t\in[t_{*j},\infty).
    \]
    By \eqref{eq:Kosc_exp_decay} the RHS is in particular bounded for $t\geq0$, which implies that $\|k_{s^j}\|_2^2$ is also bounded.
    Moreover, since the RHS is bounded by $C_je^{-\frac{1}{2}ct}$, multiplying $e^{ct}$ yields $(e^{ct}\|k_{s^j}\|_2^2)'\leq C_je^{\frac{1}{2}ct}$, and hence integrating over $[\frac{t}{2},t]\subset[t_{*j},\infty)$ gives
    \[
    e^{ct}\|k_{s^j}\|_2^2(t)-e^{\frac{1}{2}ct}\|k_{s^j}\|_2^2(\tfrac{t}{2}) \leq C_j e^{\frac{1}{2}ct}.
    \]
    Therefore, $\|k_{s^j}\|_2^2\leq C_je^{-\frac{1}{2}ct}$ holds for all $t\geq0$ and $j\geq1$.
    This combined with \eqref{eq:PW_infty} and $L\leq L_0$ yields \eqref{eq:main_3} and \eqref{eq:main_4}:
    \[
    \|k-\bar{k}\|_\infty^2\leq Ce^{-\frac{1}{2}ct}, \quad \|k_{s^j}\|_\infty^2\leq C_je^{-\frac{1}{2}ct} \quad(j\geq1).
    \]
    

    Finally, we estimate the center of mass of the curve $\gamma$.
    For any field $f$ let $\bar{f}:=\frac{1}{L}\int f ds$ denote the average.
    Using $F:=(-1)^{m+1}(k-\bar{k})_{s^{2m}}$, we compute
    \begin{align*}
        \frac{d}{dt}\bar{\gamma} &= -\frac{L'}{L^2}\int \gamma ds + \frac{1}{L}\int \gamma_t ds + \frac{1}{L}\int \gamma ds_t \\
        &=-\frac{\overline{Fk}}{L}\int \gamma ds - \overline{F\nu} + \frac{1}{L}\int \gamma Fkds\\
        &= \frac{1}{L}\int (Fk-\overline{Fk}) \gamma ds -\overline{F\nu}\\
        &= \frac{1}{L}\int (Fk-\overline{Fk}) (\gamma -\bar{\gamma}) ds -\overline{F\nu},
    \end{align*}
    and hence by using $\|\gamma-\bar{\gamma}\|_\infty \leq L$,
    \begin{align*}
        \left|\frac{d}{dt}\bar{\gamma}\right| \leq \|Fk-\overline{Fk}\|_\infty\|\gamma-\bar{\gamma}\|_\infty + \|F\|_\infty \leq 2\|F\|_\infty\|k\|_\infty L + \|F\|_\infty.
    \end{align*}
    Since the decay of $\|k-\bar{k}\|_\infty$ implies the boundedness of
    \[
    L\|k\|_\infty\leq L\|k-\bar{k}\|_\infty + L\bar{k} \leq L_0\|k-\bar{k}\|_\infty + 2\pi N_0,
    \]
    and since $\|F\|_\infty^2\leq Ce^{-\frac{1}{2}ct}$, we obtain $|\frac{d}{dt}\bar{\gamma}|\leq Ce^{-\frac{1}{4}ct}$.
    Therefore, there exists $P_\infty:=\lim_{t\to\infty}\bar{\gamma}(t)\in\R^2$ and it satisfies
    \[
    |\bar{\gamma}(t)-P_\infty|\leq \int_t^\infty\left|\frac{d}{dt}\bar{\gamma}(t')\right|dt'\leq Ce^{-\frac{1}{4}ct}.
    \]
    This shows \eqref{eq:main_5}, completing the proof.
\end{proof}

We close this subsection by briefly mentioning previous works.
As emphasized in the introduction, all the previous works are crucially based on the assumption $A_0\neq0$, which we will no longer mention below.

In the case of $m=0$ (with general $N_0\geq1$) and of $N_0=1$ (with general $m\geq0$), similar exponential convergence results are obtained by Nagasawa--Nakamura \cite{NN19,NN21} and Nakamura \cite{nakamura2019large}, respectively.
Their proof is different but also based on interpolation estimates.

For the $m=1$ case \eqref{eq:SDF} (with general $N_0\geq1$), Chou \cite{Chou03} sketched a proof of smooth convergence of immortal solutions.
This proof, combined with additional arguments for uniform bounds, is valid at least for a time subsequence $t_j\to\infty$.
Then Wheeler \cite{Wheeler22} discussed exponential convergence in the full limit $t\to\infty$, but unfortunately the proof contains several technical gaps, besides the point discussed in Remark \ref{rem:weakly_non-convex}.
[The author is announced by Wheeler that errata of \cite{Wheeler22} will be published, still assuming $A_0\neq0$.]

\subsection{Existence of immortal solutions}\label{subsec:existence_immportal}

In this final section we prove Theorem \ref{thm:rotationary_symmetric}.
This is now an easy consequence of already obtained results, together with the following key fact that the parametric isoperimetric inequality holds for any curve with suitable rotational symmetry.

\begin{theorem}[{\cite[Theorem 1.1]{MO21}}]\label{thm:isoper}
    Let $n\geq N\geq1$ and $\gamma\in A_{N,n}$. 
    Then
    \[
    L[\gamma]^2\geq 4\pi N[\gamma]A[\gamma].
    \]
\end{theorem}

\begin{proof}[Proof of Theorem \ref{thm:rotationary_symmetric}]
    Notice that by assumption $N_0,A_0>0$.
    By using unique solvability we deduce that the property being an element of $A_{N_0,n}$ is preserved along the flow (cf.\ \cite[Appendix C]{MO21}).
    Then by Theorem \ref{thm:isoper} we deduce that
    \[
    \frac{L_0}{\inf_{t\in[0,T)}L(t)}\leq \frac{L_0}{\sqrt{4\pi N_0A_0}}\leq \sqrt{1+\varepsilon_2}.
    \]
    Hence, Lemma \ref{lem:global_existence} with $t_0=0$ implies $T=\infty$ if the small constants $\varepsilon_1,\varepsilon_2>0$ in the assumption are suitably chosen.
\end{proof}

Now we review and compare Theorem \ref{thm:rotationary_symmetric} with known results.
The situation is significantly different between the cases of $m=0$ and $m\geq1$, and also between $N_0=1$ and $N_0\geq2$.

We first address the $m=0$ case \eqref{eq:ACSF}.
This case has the particular property that the curvature satisfies a certain maximum principle, so that the local convexity (the positivity of $k$) is preserved along the flow.
For the solution curve itself, however, such a strong maximum principle does not hold, and indeed the embeddedness is not preserved \cite{MS00}.

For embedded curves, Gage \cite{Gage86} proved if the initial curve is convex, then the solution is immortal and converges to a circle (see also \cite{CLW13}).
Gage's result is recently extended to a nontrivial class of non-convex embedded curves \cite{Dittberner21,Dittberner2022distance,GP23}.
In general it is expected that finite time blowup occurs even for embedded initial curves \cite{Gage86,Mayer01}.
For immersed curves, Wang--Kong \cite{WK14} extended Gage's result to locally convex curves with $N_0\geq2$ and $A_0>0$ under rotational symmetry.
Non-convex initial curves with $N_0\geq2$ are newly addressed in Theorem \ref{thm:rotationary_symmetric} (besides trivial examples of multiply-covered embedded curves).

We then turn to the $m\geq1$ case, where equations \eqref{eq:m-ACF} are of higher order.
It is a common property of higher order flows to lose many positivity properties (including convexity) due to the lack of maximum principles \cite{GI98,GI99,MS00,EM01,Blatt10,MMR23}.

However, if $N_0=1$, then a circle is dynamically stable and many kinds of stability results are known.
More precisely, if the initial curve is close to a circle, then the solution is immortal and converges to a circle.
See \cite{EG97,EMS98,EscherMucha10,W13,LSS20,LeCroneSimonett20} for $m=1$, and \cite{PW16} for $m\geq1$.
In stark contrast, if $N_0\geq2$, then circles are generally unstable and small perturbations can lead to finite time blowup.
This issue can be overcome by assuming rotational symmetry, as suggested by Chou \cite{Chou03} and precisely discussed by the author and Okabe \cite{MO21}.
Theorem \ref{thm:rotationary_symmetric} shows that this idea extends to any order $m\geq0$, thus confirming the power of the parametric isoperimetric inequality in a wide class of area-preserving geometric flows.

Finally, we discuss some open problems, in addition to the well-known Gage--Mayer conjecture for \eqref{eq:ACSF} \cite{Gage86,Mayer01}, and Giga's and Chou's conjecture for \eqref{eq:SDF} \cite{Wheeler22}.
The first problem is about nontrivial immortal solutions.
Theorem \ref{thm:rotationary_symmetric} requires almost circularity, but we expect that there also exist immortal solutions that are initially far from being circular, even under embeddedness.

\begin{conjecture}
    Let $m\geq1$.
    For any given parameter $L_0\geq \sqrt{4\pi}$, there is an embedded initial curve $\gamma_0$ of length $L_0$ and enclosed area $1$ such that the solution to \eqref{eq:m-ACF} starting from $\gamma_0$ is immortal.
\end{conjecture}

Note that in the $m=0$ case \eqref{eq:ACSF}, this property is indeed true thanks to Gage's general theory for convex curves \cite{Gage86} together with the fact that there is a continuous family of ellipses interpolating a round circle and an arbitrarily thin ellipse.
Multiple covers of such a family give explicit examples for Remark \ref{rem:blowup_parameter}.
On the other hand, for $m\geq1$, the flow does not preserve convexity and we require a substantially new technique that departs from the nearly circular regime.

Next we consider blowup solutions.
Although immortal solutions need to become circular, the circles are unstable for $N_0\geq2$.
Hence it would be natural to formulate the following

\begin{conjecture}
    Let $m\geq0$ and $N_0\geq2$.
    Then generic initial data lead to finite-time blowup solutions to \eqref{eq:m-ACF}.
\end{conjecture}

Here the term ``generic'' means for example that for any given curve there is an arbitrary small (smooth) perturbation from which the solution blows up.

\bibliography{SDF}

\end{document}